\documentclass{amsart}
\usepackage{graphicx}
\usepackage[utf8]{inputenc}

\usepackage{amssymb}
\usepackage{latexsym}
\usepackage{amsfonts}
\usepackage{amsmath}
\usepackage{amsthm}
\usepackage{amsbsy}
\usepackage{bm}
\usepackage{hyperref}
\pdfoutput=1

\calclayout

\DeclareMathOperator{\Prod}{Prod}
\DeclareMathOperator{\Aut}{Aut}
\DeclareMathOperator{\Rep}{Rep}
\DeclareMathOperator{\Diag}{Diag}

\DeclareMathOperator{\All}{All}

\DeclareMathOperator{\Num}{Num}

\DeclareMathOperator{\Inj}{Inj}
\DeclareMathOperator{\Sym}{Sym}

\renewcommand{\leq}{\leqslant}
\renewcommand{\geq}{\geqslant}

\newcommand{\J}{\mathcal J}

\theoremstyle{plain}
\newtheorem{lemma}{Lemma}

\newtheorem{theorem}[lemma]{Theorem}
\newtheorem{proposition}[lemma]{Proposition}

\theoremstyle{definition}
\newtheorem{definition}[lemma]{Definition}
\newtheorem{example}[lemma]{Example}

\newtheorem{question}[lemma]{Question}

\numberwithin{equation}{section}
\numberwithin{lemma}{section}

\begin{document}

\title{Alphabet-Almost-Simple $2$-Neighbour Transitive Codes}
 \author{Neil I Gillespie and Daniel R Hawtin}
 
 \address{[Gillespie] Heilbronn Institute for Mathematical Research, School of Mathematics, Howard House, University of Bristol, BS8 1SN, United Kingdom.\newline
 \indent[Hawtin] Centre for the Mathematics of Symmetry and Computation, University of Western Australia, 35 Stirling Highway, Crawley, WA 6009, Australia.}

\email{neil.gillespie@bristol.ac.uk \\ daniel.hawtin@research.uwa.edu.au}

\date{\today}

\thanks{
{\it $2$-neighbour transitive \quad alphabet almost simple \quad automorphism groups \quad Hamming graph \quad completely transitive}\\
The second author is supported by an Australian Postgraduate Award and UWA Top-up Scholarship.}

\begin{abstract}
 Let $X$ be a subgroup of the full automorphism group of the Hamming graph $H(m,q)$, and $C$ a subset of the vertices of the Hamming graph. We say that $C$ is an \emph{$(X,2)$-neighbour transitive code} if $X$ is transitive on $C$, as well as $C_1$ and $C_2$, the sets of vertices which are distance $1$ and $2$ from the code. This paper begins the classification of $(X,2)$-neighbour transitive codes where the action of $X$ on the entries of the Hamming graph has a non-trivial kernel. There exists a subgroup of $X$ with a $2$-transitive action on the alphabet; this action is thus almost-simple or affine. If this $2$-transitive action is almost simple we say $C$ is \emph{alphabet-almost-simple}. The main result in this paper states that the only alphabet-almost-simple $(X,2)$-neighbour transitive code with minimum distance $\delta\geq 3$ is the repetition code in $H(3,q)$, where $q\geq 5$.
\end{abstract}

\maketitle

\section{Introduction}

Ever since Shannon's 1948 paper \cite{shannon} there has been a great deal of interest around families of error-correcting codes with a high degree of symmetry. The rationale behind this interest is that codes with symmetry should have good error correcting properties. The first families classified were perfect (see \cite{tietavainen1973nonexistence} or \cite{Zinoviev73thenonexistence}) and  nearly-perfect (defined in \cite{GOETHALS197265} classified in \cite{LINDSTROM197740}) codes over prime power alphabets. Such codes are rare. In an effort to find further classes of efficient codes Delsarte \cite{delsarte1973algebraic} introduced \emph{completely regular} codes, a more general class of codes that posses a high degree of combinatorial symmetry. Much effort has been put into classifying particular classes of completely regular codes (see for instance \cite{borges2000nonexistence,borges2013new}), and new completely regular codes continue to be found \cite{2014arXiv1410.4785G}.  However, completely regular codes have proven to be hard to classify, and this remains an open problem. 

\emph{Completely transitive} (first defined in \cite{sole1987completely} with a generalisation studied in \cite{michealmast}) codes are a class of codes with a high degree of algebraic symmetry and are a subset of completely regular codes. As such a classification of completely transitive codes would be interesting from the point of view of classifying completely regular codes. This problem also remains open. 

Here, we relax the conditions of complete transitivity and study the family of \emph{$2$-neighbour transitive codes}, a class of codes with a moderate degree of algebraic symmetry. Note that every completely transitive code (see Section~\ref{prelim}) is $2$-neighbour transitive. By studying this class of codes we hope to find new codes and gain a better understanding of completely transitive codes. Indeed a classification of $2$-neighbour transitive codes would have as a corollary a classification of completely transitive codes. We also note that codes with $2$-transitive actions on the entries of the Hamming graph (which $2$-neighbour transitive codes indeed have), have been of interest lately, where this fact can be used to prove that certain families of codes achieve capacity on erasure channels \cite{2016arXiv160104689K}. The analysis of $2$-neighbour transitive codes is being attacked as three separate problems: \emph{entry-faithful} (see \cite{ef2nt}), \emph{alphabet-almost-simple}, and \emph{alphabet-affine}. This paper concerns the alphabet-almost-simple case. The results of this paper do not return any new examples.

However, the results here are of interest from the point of view of perfect codes over an alphabet of non-prime-power size, since in this case a code cannot be alphabet-affine (and also not entry-faithful, by \cite{ef2nt}), but may be alphabet-almost-simple. The existence of perfect codes over non-prime-power alphabets with covering radius $1$ or $2$, is still an open question (see \cite{hong1984}). By Theorem~\ref{aasmain}, if such codes exist, then they cannot be $2$-neighbour transitive (unless they are equivalent to the repetition code of length $3$). Note that in the prime power case, for each set of parameters for which a perfect code with covering radius $\rho\geq 2$ exists, a $2$-neighbour transitive code with those parameters exists. That is, the repetition and Golay codes are $2$-neighbour transitive. In fact, the repetition, Hamming and Golay codes are completely transitive, by \cite[Section~3.5]{michealmast}.

\subsection{Statement of the main results}

Let $X$ be a subgroup of the full automorphism group of the Hamming graph $\varGamma=H(m,q)$ and $C$ be a code, that is, a subset of the set of vertices $V\varGamma$. We say that $C$ is an \emph{$(X,s)$-neighbour transitive code} if $X$ is transitive on $C=C_0$, $C_1,\ldots,C_s$ (where $C_i$ are parts of the distance partition, see Section~\ref{prelim}). In joint work with Giudici and Praeger \cite{ef2nt}, the authors classified all $(X,2)$-neighbour transitive codes for which the group $X$ acts faithfully on the set of entries of the Hamming graph. In this paper, we begin the study of $(X,2)$-neighbour transitive codes such that the action of $X$ on the entries has a non-trivial kernel.

If $C$ is an $(X,2)$-neighbour transitive code with minimum distance $\delta\geq 3$, then $X_1$, the subgroup of $X$ which fixes the first entry of $H(m,q)$, has a $2$-transitive action on the alphabet in that entry (see \cite[Proposition 2.7]{ef2nt}). Any $2$-transitive action is of affine or almost-simple type \cite[Theorem~4.1B]{dixon1996permutation}. If $C$ is $(X,s)$-neighbour transitive, the action of $X_1$ on the alphabet is almost-simple and the action $X$ on the entries is transitive, we say $C$ is \emph{alphabet-almost-simple $(X,s)$-neighbour transitive}. Our main aim here is to prove the non-existence of codes which are alphabet-almost-simple $(X,2)$-neighbour transitive with minimum distance $\delta\geq 4$.

\begin{theorem}\label{aasmain}
 Let $C$ be an alphabet-almost-simple $(X,2)$-neighbour transitive code in $H(m,q)$ with minimum distance $\delta \geq 3$. Then $\delta=3$ and $C$ is equivalent to the repetition code in $H(3,q)$, where $q\geq 5$.
\end{theorem}

In Section~\ref{prelim} we define the notation used in the paper. In Section~\ref{structures} we give some results on the structure of alphabet-almost-simple $(X,2)$-neighbour transitive codes, as well as pose some questions about codes for which the action of $X_1$ on the alphabet in the first entry is affine. We present some examples of codes with properties of interest in relation to our results in Section~\ref{aasexamples}. Finally, in Section~\ref{aasclassification}, we give a classification of \emph{diagonally $(X,2)$-neighbour transitive} codes (see Definition~\ref{diagxsnt}) and prove Theorem~\ref{aasmain}.

\section{Preliminaries}\label{prelim}

Throughout this paper we let $M=\{1,\ldots,m\}$ and $Q=\{1,\ldots,q\}$, with $m,q\geq 2$, though if $q=2$ we will at times use $Q=\{0,1\}$. We refer to $M$ as the \emph{set of entries} and $Q$ as the \emph{alphabet}. The vertex set of the Hamming graph $\varGamma=H(m,q)$ consists of all $m$-tuples with entries labeled by the set $M$, taken from the set $Q$. An edge exists between two vertices if they differ as $m$-tuples in exactly one entry. For vertices $\alpha,\beta$ of $H(m,q)$ the \textit{Hamming distance} $d(\alpha,\beta)$ is the number of entries in which $\alpha$ and $\beta$ differ, i.e. the usual graph distance in $\varGamma$.

A code $C$ is a subset of the vertex set of the Hamming graph. The \emph{minimum distance} of $C$ is $\delta=\min\{d(\alpha,\beta)\mid \alpha,\beta\in C,\alpha\neq \beta\}$. For a vertex $\alpha\in H(m,q)$, define $$\varGamma_r(\alpha)=\{\beta\in\varGamma \mid d(\alpha,\beta)=r\},\quad \text{and}\quad d(\alpha,C)=\min\{d(\alpha,\beta) \mid \beta\in C\}.$$  We then define the \textit{covering radius} to be $$\rho =\max\{d(\alpha,C)\mid\alpha\in\varGamma\}.$$ For any $r\leq \rho$, define $C_r=\{\alpha\in\varGamma \mid d(\alpha,C)=r\}$. Note that $C_i$ is the disjoint union $\cup_{\alpha\in C}\varGamma_i(\alpha)$ for $i\leq \lfloor\frac{\delta-1}{2}\rfloor$. 

\subsection{Automorphism groups}

The automorphism group $\Aut(\varGamma)$ of the Hamming graph is the semi-direct product $B\rtimes L$, where $B\cong S_q^m$ and $L\cong S_m$ (see \cite[Theorem 9.2.1]{brouwer}). We refer to $B$ as the \emph{base group}, and $L$ as the \emph{top group}, of $\Aut(\varGamma)$. Let $g=(g_1,\dots,g_m)\in B$, $\sigma\in L$ and $\alpha$ be a vertex in $H(m,q)$. Then $g$ and $\sigma$ act on $\alpha$ as follows: 
\begin{equation*}
\alpha^g =(\alpha_1^{g_1},\ldots,\alpha_m^{g_m})\quad\text{and}\quad
\alpha^\sigma=(\alpha_{1{\sigma^{-1}}},\ldots,\alpha_{m{\sigma^{-1}}}).
\end{equation*}

We define the automorphism group of a code $C$ in $H(m,q)$ to be $\Aut(C)=\Aut(\varGamma)_C$, the setwise stabiliser of $C$ in $\Aut(\varGamma)$. For a subgroup $X\leq \Aut(\varGamma)$ we define two other important actions of $X$ which will be useful to us. First, consider the action of $X$ on the set of entries $M$, which we will write as $X^M$, defined by the following homomorphism:
\begin{center}
\begin{tabular}{cccc}
 $\mu$ :& $X$ &$\longrightarrow$ & $S_m$\\
& $(h_1,\ldots,h_m)\sigma$ &$\longmapsto$ & $\sigma$ 
\end{tabular}
\end{center}
We define $K$ to be the kernel of this map and note that $K=X\cap B$. In this paper we are concerned with $(X,2)$-neighbour transitive codes where $K\neq 1$.

We also consider the action of the stabiliser $X_i\leq X$ of the entry $i\in M$, on the alphabet $Q$. We denote this action by $X_i^Q$ and it is defined by the homomorphism:
\begin{center}
\begin{tabular}{cccc}
 $\varphi_i$ :& $X_i$ &$\longrightarrow$ & $S_q$\\
& $(h_1,\ldots,h_m)\sigma$ &$\longmapsto$ & $h_i$ 
\end{tabular}
\end{center}

Let $C$ be a code in $H(m,q)$ and let $X$ be a subgroup of $\Aut(\varGamma)$. Recall that $C$ is \emph{$(X,s)$-neighbour transitive} if each $C_i$ is an $X$-orbit for $i=0,\ldots,s$. Note that this implies $X\leq \Aut(C)$ and $C$ is also $(X,r)$-neighbour transitive, for $r<s$. If $s=1$ then $C$ is simply \emph{$X$-neighbour transitive} and if $s=\rho$, the covering radius, then $C$ is \emph{$X$-completely transitive}. Recall, if $C$ is $(X,s)$-neighbour transitive, $X^M$ is transitive on $M$ and the group $X_1^Q$ is almost-simple, then we say $C$ is \emph{alphabet-almost-simple $(X,s)$-neighbour transitive}. We may sometimes omit the group $X$ from any of these terms if the meaning is clear from the context.


We say that two codes, $C$ and $C'$, in $H(m,q)$, are \textit{equivalent} if there exists $x\in \Aut(\varGamma)$ such that $C^x=C'$. Since elements of $\Aut(\varGamma)$ preserve distance, equivalence preserves minimum distance. 

\subsection{Projections}

For $\alpha\in \varGamma$, we refer to the element of $Q$ appearing in the $i$-th entry of $\alpha$ as $\alpha_i$, so that $\alpha=(\alpha_1,\ldots,\alpha_m)$. For a subset $J=\{j_1,\ldots,j_k\}\subseteq M$ we define the \emph{projection of $\alpha$ with respect to $J$} as $\pi_J(\alpha)=(\alpha_{j_1},\ldots,\alpha_{j_k})$. For a code $C$ we then define the \emph{projection of $C$ with respect to $J$} as $\pi_J(C)=\{\pi_J(\alpha)\mid \alpha\in C\}$. So $\pi_J$ maps a vertex or code from $H(m,q)$ into the smaller Hamming graph $H(k,q)$.

Let $X_J$ be the setwise stabiliser of a subset $J=\{j_1,\ldots,j_k\}\subseteq M$. For $x=(h_1,\ldots,h_m)\sigma\in X_J$, we define the \emph{projection of $x$ with respect to $J$} as $\chi_J(x)$ where $$\pi_J(\alpha)^{\chi_J(x)}=\pi_J(\alpha^x).$$ To be well defined, this requires $x\in X_J$ and it follows that $\chi_J(x)=(h_{j_1},\ldots,h_{j_k})\hat\sigma\in\Aut(H(k,q))$, where $\hat\sigma$ is the element of $\Sym(J)$ induced by $\sigma$. Moreover, we define $\chi_J(X)=\{\chi_J(x)\mid x\in X_J\}$.

\section{Structural results}\label{structures}

We collect below some results from \cite{gillespieCharNT}, where alphabet-almost-simple $X$-neighbour transitive codes with $\delta\geq 3$ are characterised. This is our starting point when looking at codes $C$ which are alphabet-almost-simple $(X,2)$-neighbour transitive with $\delta\geq 3$, since we then have that $C$ is indeed $X$-neighbour transitive.

For a subgroup $T\leq S_q$ we define $\Diag_m(T)=\{(h,\ldots,h)\in B\mid h\in S_q\}$.
\begin{definition}\label{diagxsnt}
 A code $C$ in $H(m,q)$ is \emph{diagonally $(X,s)$-neighbour transitive}, if $C$ is $(X,s)$-neighbour transitive and $X\leq \Diag_m(S_q)\rtimes L$.
\end{definition}

%

\begin{proposition}\label{partimpdiag}
  Let $C$ be an alphabet-almost-simple $X$-neighbour transitive code with $\delta\geq 3$. Then there exists an $X$-invariant partition $\J=\{J_1,\ldots,J_\ell\}$ of $M$ such that $\pi_{J_i}(C)$ is diagonally $\chi_{J_i}(X)$-neighbour transitive and $\delta(\pi_{J_i}(C))\geq 2$.
\end{proposition}

\begin{proof}
 Let $T$ be the non-abelian simple socle of the almost-simple $2$-transitive group $X_1^Q$. By the discussion following \cite[Proposition~5.2]{gillespieCharNT}, there exists a partition $\J=\{J_1,\ldots,J_\ell\}$ of $M$, with $|J_i|=k$, such that the socle of $X\cap B$ is equal to $D_1\times \cdots \times D_\ell$, where each $D_i$ is a full diagonal subgroup of $T^k$ acting on $\pi_{J_i}(\varGamma)$. Moreover, by \cite[Remark~5.4]{gillespieCharNT}, $\J$ is $X$-invariant. By examining this socle, it can be shown \cite[Section~5]{gillespieCharNT} that, up to equivalence, two possibilities occur. Either $\chi_{J_i}(X)\leq \Diag_k(S_q)\rtimes S_k$ for all $i$, or there exists a more refined $X$-invariant partition $\hat{\mathcal{J}}$ of $M$ such that $\chi_J(X)\leq\Diag_{\hat{k}}(S_q)\rtimes S_{\hat{k}}$ for all $J\in\hat{\mathcal{J}}$.
 
 
 In either case, it follows from \cite[Prop.~3.4 and Cor.~3.7]{gillespieCharNT} that $\chi_{J_i}(X)$ acts transitively on $\pi_{J_i}(C)$ and either $\pi_{J_i}(C)$ is the complete code or it is $\chi_{J_i}(X)$-neighbour transitive with minimum distance at least $2$. Since $\chi_{J_i}(X)$ is a diagonal subgroup, we deduce that $\pi_{J_i}(C)$ is $\chi_{J_i}(X)$-neighbour transitive as no diagonal subgroup acts transitively on the complete code.
\end{proof}

\begin{proposition}\label{partx2nt}
 Let $C$ be an $(X,2)$-neighbour transitive code with $\delta\geq 3$ in $H(m,q)$, and suppose $\J=\{J_1,\ldots,J_l\}$ is an $X$-invariant partition of $M$. Then for all $i\in\{1,\ldots,l\}$, either;
 \begin{enumerate}
  \item $\pi_{J_i}(C)$ is the complete code, $\delta(\pi_{J_i}(C))=1$, and $\chi_{J_i}(X)$ is transitive on $\pi_{J_i}(C)$;\label{partx2nt1}
  \item $\pi_{J_i}(C)$ has covering radius $1$, $\delta(\pi_{J_i}(C))= 2$ or $3$, and is $(\chi_{J_i}(X),1)$-neighbour transitive; or\label{partx2nt2}
  \item $\pi_{J_i}(C)$ is $(\chi_{J_i}(X),2)$-neighbour transitive.\label{partx2nt3}
 \end{enumerate}
\end{proposition}

\begin{proof}
 Let ${\bar C} =\pi_{J_i}(C)$. The fact that $\chi_{J_i}(X)$ is transitive on ${\bar C}$ and ${\bar C}_1$, if $C_1$ is non-empty, follows from \cite[Proposition~3.4]{gillespieCharNT}. From this we deduce both parts \ref{partx2nt1}~and \ref{partx2nt2}~hold. Now, \cite[Corollary~3.7]{gillespieCharNT} gives us that $\delta(\pi_{J_i}(C))\geq 2$, and $\delta(\pi_{J_i}(C))$ is at most $3$ in part \ref{partx2nt2}. Moreover, to prove part \ref{partx2nt3}, we need only show that if ${\bar C}_2$ is non-empty, then $\chi_{J_i}(X)$ is transitive on ${\bar C}_2$. 
 
 Suppose $\bar C$ has covering radius at least $2$. Let $\mu,\nu\in {\bar C}_2$. Then there exists $\alpha,\beta\in C$ such that $d(\mu,\pi_{J_i}(\alpha))=d(\nu,\pi_{J_i}(\beta))=2$. Let $\hat\nu\in H(m,q)$ with $\hat\nu_u=\nu_u$ for $u$ in $J_i$ and $\hat\nu_v=\alpha_v$ otherwise. Similarly, let $\hat\mu\in H(m,q)$ with $\hat\mu_u=\mu_u$ for $u$ in $J_i$ and $\hat\mu_v=\beta_v$ otherwise. We claim that $\hat\nu, \hat\mu \in C_2$. We show this for $\hat\nu$ and note that an identical agrument holds for $\hat\mu$. First, note that $d(\alpha,\hat\nu)=2$ and $\delta\geq 3$, so $\hat\nu\notin C$. Suppose $\hat\nu\in C_1$. Then there exists $\alpha'\in C$ such that $d(\hat\nu,\alpha')=1$. We then have $d(\nu,\pi_{J_i}(\alpha'))\leq 1$. However, this contradicts $\nu\in {\bar C}_2$. Hence $\hat\mu,\hat\nu \in C_2$. 
 
 As $C$ is $(X,2)$-neighbour transitive, there exists an $x=h\sigma\in X$ mapping $\hat\nu$ to $\hat\mu$. We claim $x\in X_{J_i}$. Suppose $x\notin X_{J_i}$. Then, since $\J$ is a system of imprimitivity for the action of $X$ on $M$, there exists $j\in\{1,\ldots,l\}$ such that $j\neq i$ and $J_j^\sigma=J_i$. Since $\pi_{J_j}(\hat\nu)=\pi_{J_j}(\alpha)$, this implies that $\pi_{J_i}(\hat\nu^x)=\pi_{J_i}(\alpha^x)\in {\bar C}$ and hence $\pi_{J_i}(\hat\nu^x)\neq \mu$, which contradicts the fact that $\hat\nu^x=\hat\mu$. Thus $x\in X_{J_i}$ and 
 $$\nu^{\chi_{J_i}(x)} = \pi_{J_i}(\hat\nu)^{\chi_{J_i}(x)} = \pi_{J_i}(\hat\nu^x)= \pi_{J_i}(\hat\mu)= \mu.$$
\end{proof}

\begin{proposition}
 Let $C$ be an $(X,2)$-neighbour transitive code in $H(m,q)$ with $\delta\geq 3$, and $\J$ be an $X$-invariant partition of $M$. Then, for all $J\in\J$,
 \begin{enumerate}
  \item $\chi_J(X)_1^Q$ is $2$-transitive on $Q$; and,
  \item for $\alpha\in C$, $\chi_J(X)_{\pi_J(\alpha)}$ is transitive on $J$.
 \end{enumerate}
\end{proposition}

\begin{proof}
 As C is $X$-neighbour transitive with $\delta\geq 3$, we have that $X_1^Q$ is $2$-transitive and $X^M$ is transitive. One then deduces that $X_i^Q$ is 2-transitive for all $i$. Now, because $\mathcal{J}$ is an $X$-invariant partition, it follows that $X_i=(X_J)_i$ for all $i\in J$. This in turn implies that $\chi_J(X)_i=\chi_J(X_i)$. It is now straight forward to show that $\chi_J(X_i)^Q=X_i^Q$.

 Now, since $X_\alpha$ is transitive on $M$ and $\mathcal{J}$ is an $X$-invariant partition of $M$, it follows that $(X_\alpha)_J$ is transitive on $J$. Thus $\chi_J(X_\alpha)\leq\chi_J(X)_{\pi(\alpha)}$ is transitive on $J$.
\end{proof}


The previous two propositions suggest studying $(X,2)$-neighbour transitive codes where $X$ acts primitively on $M$ with $\delta\geq 2$. An answer to the following questions would provide us with the building blocks for $(X,2)$-neighbour transitive codes with $\delta\geq 3$.

\begin{question}\label{ntrdeltag2}
 Can we classify all $(X,2)$-neighbour transitive codes with $\delta\geq 2$ such that $X^M$ is primitive and $X_1^Q$ is $2$-transitive?
\end{question}

\begin{question}\label{perfectcon}
 Can we classify all $(X,1)$-neighbour transitive codes with $\delta=2$ or $3$ and $\rho=1$ such that $X^M$ is primitive and $X_1^Q$ is $2$-transitive?
\end{question}

If $C$ is $(X,s)$-neighbour transitive and $X$ acts faithfully on $M$ we say $C$ is \emph{entry-faithful $(X,s)$-neighbour transitive}. If $C$ is $(X,s)$-neighbour transitive, $X\leq \Aut(C)$, $X^M$ is transitive, and $X_1^Q$ is affine we say $C$ is \emph{alphabet-affine $(X,s)$-neighbour transitive}. Questions~\ref{ntrdeltag2} and \ref{perfectcon} can be further broken down into entry-faithful and non-trivial kernel cases, that is, alphabet-affine and alphabet-almost-simple. By the main result of this paper, the outstanding cases of Question~\ref{ntrdeltag2} are alphabet-almost-simple $(X,2)$-neighbour transitive with $\delta=2$, and alphabet-affine $(X,2)$-neighbour transitive, where $X^M$ is primitive and $X_1^Q$ is $2$-transitive. 

Given Proposition~\ref{partimpdiag}, a third question is the following.

\begin{question}
 Can we contrsuct $(X,2)$-neighbour transitive codes with $\delta\geq 3$ by taking copies of $(X,1)$-neighbour transitive codes with $\delta=2$ or $3$ and $\rho=1$. 
\end{question}

\section{Examples}\label{aasexamples}

We begin this section by considering some examples of codes which have properties relating to the results of the previous section. We first introduce the operators $\Prod$ and $\Rep$ which allow the construction of new codes from old ones. For an arbitrary code $C$ in $H(m,q)$ we define $\Prod(C,\ell)$ and $\Rep_\ell(C)$ in $H(m\ell,q)$ as $$\Prod(C,\ell)=\{(\boldsymbol\alpha_1,\ldots,\boldsymbol\alpha_\ell)\mid \boldsymbol\alpha_i\in C\},$$ and $$\Rep_\ell(C)=\{(\boldsymbol\alpha,\ldots,\boldsymbol\alpha)\mid \boldsymbol\alpha\in C\}.$$ The \emph{repetition code} $\Rep(m,q)$ in $H(m,q)$ is the set of all vertices $(a,\ldots,a)$ consisting of a single element $a\in Q$ repeated $m$ times.

The next two examples are codes which are alphabet-almost-simple $X$-completely transitive, though the second has $\delta=2$.

\begin{example}\label{k3}
 Let $C=\Rep(3,q)$, where $q\geq 5$, and $X=\Diag_3(S_q)\rtimes S_3$. By \cite[Example~3.1]{Giudici1999647} $C$ is $X$-completely transitive with covering radius $\rho=2$, and hence $(X,2)$-neighbour transitive. (Note that the proof of this fact is stated for $q\geq 7$, but works for $q\geq 5$.) Now $X_1^Q\cong S_q$, with $q\geq 5$, is almost-simple and $X^M\cong S_3$ is transitive on $M$. Hence $C$ is alphabet-almost-simple $X$-completely transitive.
\end{example}

\begin{example}\label{k2}
 Let $q\geq 5$, $\ell\geq 2$, $C=\Prod(\Rep(2,q),\ell)$ and $X=(\Diag_2(S_q))^\ell\rtimes U$, where $\Diag_2(S_q)$ is a subgroup of the base group of $\Aut(H(2,q))$ and $U=S_2\wr S_\ell=S_2^\ell\rtimes S_\ell$ is a subgroup of the top group of $\Aut(H(2\ell,q))$. Let $\J=\{J_1,\ldots,J_\ell\}$, with $J_i=\{2i-1,2i\}$, be the partition of $M$ preserved by $U$. Note that $\delta=2$. Let $R\subseteq\{1,\ldots,\ell\}$ of size $s$, and $\nu\in H(m,q)$ be such that $\pi_{J_i}(\nu)=(a,b)$, where $a\neq b$ for all $i\in R$, and $a=b$ for all $i\notin R$. Any codeword $\beta$ is at least distance $s$ from $\nu$, since $d(\pi_{J_i}(\nu),\pi_{J_i}(\beta))\geq 1$ for each $i\in R$. Also, there exists some codeword $\alpha$ with $\pi_{J_i}(\alpha)=(a,a)$ whenever $\pi_{J_i}(\nu)=(a,b)$ for $i\in\{1,\ldots,\ell\}$, and hence $d(\alpha,\nu)=s$. So $\nu\in C_s$. Any vertex $\nu$ of $H(2\ell,q)$ can be expressed in this way, for some $R$, since $\pi_{J_i}(\nu)=(a,b)$ has either $a=b$ or $a\neq b$. Thus, for each $s$, $C_s$ consists of all such vertices $\nu$ where $|R|=s$. It also follows from this that $\rho=\ell$. 
 
 Let $\nu\in C_s$, with $R$ as above. Let $x=(h_{J_1},\ldots,h_{J_\ell})\sigma\in X$ where $h_{J_i}\in\Diag_2(S_q)$ such that $\pi_{J_i}(\nu)^{h_{J_i}}=(1,2)$, for $i\in R$, and $\pi_{J_i}(\nu)^{h_{J_i}}=(1,1)$, for all $i\notin R$. Moreover, since $S_\ell$ is $\ell$-transitive, there exists a $\sigma\in S_\ell\leq S_2\wr S_\ell$ mapping $\{J_{i_1},\ldots,J_{i_s}\}$ to $\{J_1,\ldots,J_s\}$ (where $R=\{i_1,\ldots,i_s\}$), whilst preserving order within each $J_i$. Then $\nu^x=\gamma\in C_s$, where $\pi_{J_i}(\gamma)=(1,2)$ for all $i\in \{1,\ldots,s\}$ and $\pi_{J_i}(\gamma)=(1,1)$ for all $i\notin \{s+1,\ldots,\ell\}$. Since we can map any such $\nu$ to $\gamma$, $X$ is transitive on $C_s$ for each $s\in\{1,\ldots,\ell\}$. Hence $C$ is $X$-completely transitive, and in particular $(X,2)$-neighbour transitive for $\ell\geq 2$. Since $X_1^Q\cong S_q$ and $X^M\cong S_2\wr S_\ell$ is transitive on $M$, $C$ is alphabet-almost-simple $X$-completely transitive.
\end{example}

\begin{lemma}\label{repnopart}
 Suppose $C$ is an $(X,2)$-neighbour transitive code in $H(m,q)$, with $q\geq 3$, and $\J$ is an $X$-invariant partition of $M$, such that $\pi_J(C)=\Rep(k,q)$, for all $J\in \J$ where $k=|J|$. Then either $\delta=k=2$, or $\J$ is a trivial partition.
\end{lemma}

\begin{proof}
 Let $x=(h_1,\ldots,h_m)\sigma\in X$ and $J\in \J$. By the hypothesis it follows that for all $a\in Q$, there exists $\alpha\in C$ such that $\pi_J(\alpha)=(a,\ldots,a)$. Suppose $J^\sigma=J'\in\J$. Then $\pi_{J'}(\alpha^x)=(a^{h_{i_1}},\ldots,a^{h_{i_k}})\sigma=(b,\ldots,b)$ for some $b\in Q$, that is, $a^{h_{i_s}}=a^{h_{i_t}}$ for all $i_s,i_t\in J$. In particular $\chi_J(x\sigma^{-1})=(h,\ldots,h)$ for some $h\in S_q$, and $X\leq \Diag_k(S_q)\wr U$, where $U$ is the stabiliser of $\J$ in the top group.
 
 Now suppose that $\J$ is a non-trivial partition, so $k,\ell\geq 2$. Since $C\subseteq \Prod(\Rep(k,q),\ell)$, which has minimum distance $k$, it follows that $\delta\geq k\geq 2$.
 
 Suppose $\delta\geq 3$. As $C$ is a subset of $\Prod(\Rep(k,q),\ell)$ we can replace $C$ by an equivalent code contained in $\Prod(\Rep(k,q),\ell)$ containing $\alpha=(1,\ldots,1)$ and such that $$\J=\left\{\{1,\ldots,k\},\{k+1,\ldots,2k\},\cdots,\{m-k+1,\ldots,m\}\right\}.$$ Consider,
 \begin{align*}
  \mu & = (2,3,1,1,\ldots,1,1,1,1,\ldots,1,\cdots, 1,\ldots,1)\quad \text{and} \\
  \nu & = (\underbrace{2,1,1,1,\ldots,1}_{k\text{ entries}},\underbrace{2,1,1,\ldots,1}_{k\text{ entries}},\cdots,\underbrace{1,\ldots,1}_{k\text{ entries}}).
 \end{align*}
 If $k=2$, then we claim $\mu\in C_2$. Any vertex $\beta\in\Prod(\Rep(2,q),\ell)\supseteq C$ with $d(\mu,\beta)=1$ is of the form $\gamma=(a,a,1,\ldots,1)$, where $a=2$ or $3$. However, no such $\gamma$ is an element of $C$, since each is distance $2$ from $\alpha$. If $k\geq 3$ then $\mu\in C_2$ since $d(\alpha,\mu)=2$ and there is no closer codeword as $\pi_{J_1}(\mu)\in \pi_{J_1}(C)_2$. In both cases $\nu\in C_2$ since $d(\alpha,\nu)=2$ and no codeword is closer, as $\pi_{J_i}(\nu)\in \pi_{J_i}(C)_1$ for $i=1,2$. Let $x=(h_1,\ldots,h_m)\sigma\in X$ such that $\mu^x=\nu$. We reach a contradiction here, since $h_1=h_2=\cdots=h_k=h$ cannot, assuming $k\geq 3$, map the set $\{1,2,3\}$ to either of the sets $\{1,2\}$ or $\{1\}$. In the case $k=2$, in at least one block we must map the set $\{1\}$ to $\{1,2\}$, which is not possible. Hence $2\geq\delta\geq k\geq 2$.
\end{proof}

The next example shows that it is possible to have a neighbour transitive code where $\delta\geq 3$ and the projection code for some system of imprimitivity on $M$ is the complete code. Note that this does not contradict the results from \cite{gillespiediadntc} as there is more than one system of imprimitivity present.

\begin{example}
 Let ${\bar C}=\Prod(C,\ell)$ be a code in $\varGamma=H(m,q)$, where $m=k\ell$ and $C$ is an $X$-neighbour transitive code in $H(k,q)$ where $X\cap B$ is transitive on $C$ and $\delta\geq 3$. Let $\bar X=\langle (X\cap B)^\ell,\Diag_\ell(X),S_\ell \rangle$ preserve the partition $$\J=\{\{1,\ldots,k\},\ldots,\{m-k+1,\ldots,m\}\}=\{J_1,\ldots,J_\ell\},$$ of $M$, where $\chi_J((X\cap B)^\ell)=X\cap B$ and $\chi_J(\Diag_\ell(X))=X$ for all $J\in \J$, and $S_\ell$ acts as pure permutations by the permuting blocks of $\J$ whilst preserving the order of entries within a given block. It follows that we preserve two $\bar X$-invariant partitions, $\J$ and one attained from taking the corresponding entries, by order, from each copy of $C$ as a block: $$\J'=\{\{1,k+1,\ldots,m-k+1\},\ldots,\{\ell,\ell+k,\ldots,m\}\}.$$ Given any $\alpha=(\boldsymbol\alpha_1,\ldots,\boldsymbol\alpha_\ell)\in {\bar C}$, $\boldsymbol\alpha_i\in C$, and $\beta=(\boldsymbol\beta_1,\ldots,\boldsymbol\beta_\ell)\in {\bar C}$, $\boldsymbol\beta_i\in C$ there exists an $x\in (X\cap B)^\ell$ mapping $\alpha$ to $\beta$ since $X\cap B$ is transitive on $C$. Hence $\bar X$ is transitive on $\bar C$. Given any two neighbours $\mu,\nu\in \varGamma_1(\alpha)$, where $\mu,\nu$ differ from $\alpha$ in the respective blocks $J_i$ and $J_j$, we can map $J_j$ to $J_i$ via some element $\sigma\in S_\ell$. Then, since $X_{\boldsymbol\alpha_i}$ is transitive on $\varGamma_1(\boldsymbol\alpha_i)$, there exists an element $x\in \Diag_\ell(X)$ such that $\pi_{J_i}(\nu^{\sigma x})=\pi_{J_i}(\mu)$. We can then map $\nu^{\sigma x}$ to $\mu$ via some element $h\in (X\cap B)^\ell$, where $\chi_{J_i}(h)=1$, since each $\pi_{J_t}(\nu^{\sigma x})$ and $\pi_{J_t}(\mu)$ are elements of $C$ for $t\neq i$ and $X\cap B$ is transitive on $C$. Hence $\sigma x h$ maps $\nu$ to $\mu$ and $\bar X$ is transitive on ${\bar C}_1$.
 
 When we consider the projection $\pi_J(\bar C)$ for any $J\in\J'$ we are left with the complete code. To see this, consider that for $(\boldsymbol\alpha_1,\ldots,\boldsymbol\alpha_\ell)\in {\bar C}$, $\boldsymbol\alpha_i\in C$, we may choose an arbitrary element of $C$ as $\boldsymbol\alpha_i$ for each $i$. Since $X_1^Q$ is $2$-transitive on $Q$, each element appears in the first entry for some codeword. Thus, as $\pi_J((\boldsymbol\alpha_1,\ldots,\boldsymbol\alpha_\ell))$ when $J=\{1,k+1,\ldots,m-k+1\}$ is the first entry of each $\boldsymbol\alpha_i$, we have that $\pi_J(\bar C)$ is the complete code.
\end{example}

\section{Alphabet-almost-simple \texorpdfstring{$(X,2)$}{(X,2)}-neighbour transitive codes}\label{aasclassification}

Before we prove the final results we define the codes used in this section, which first requires the following definition.

\begin{definition}\label{defnum}
 Define the \emph{composition} of a vertex $\alpha\in H(m,q)$ to be the set $$Q(\alpha)=\{(a_1,p_1),\ldots,(a_q,p_q)\},$$ where $p_i$ is the number of entries of $\alpha$ which take the value $a_i \in Q$. For $\alpha\in H(m,q)$ define the set $$\Num(\alpha)=\{(p_1,s_1),\ldots,(p_j,s_j)\},$$ where $(p_i,s_i)$ means that $s_i$ distinct elements of $Q$ appear precisely $p_i$ times in $\alpha$. 
\end{definition}

\begin{definition}\label{diagntcodes}
We define the following codes:
\begin{enumerate}
 \item $\Inj(m,q)$, where $m<q$, is the set of all vertices $\alpha\in H(m,q)$ such that $\Num(\alpha)=\{(1,m)\}$;
 \item for $m$ odd, $W([m/2],2)$ is the set of vertices in $\alpha\in H(m,2)$ such that $\Num(\alpha)=\{(m+1)/2,1),(m-1)/2,1)\}$; and,
 \item $\All(pq,q)$, where $pq=m$, is the set of all vertices $\alpha\in H(m,q)$ such that $\Num(\alpha)=\{(p,q)\}$. 
\end{enumerate}
\end{definition}
For more information on these codes see \cite[Definition 2]{gillespiediadntc}. The following lemma is \cite[Lemma~4]{gillespiediadntc}. 

\begin{lemma}\label{numpres}
For any vertex $\alpha$ of $H(m,q)$, $\Num(\alpha)$ is preserved by $\Diag_m(S_q)\rtimes L$. 
\end{lemma}

The last result, when combined with the classification of diagonally neighbour transitive codes \cite[Theorem~4.3]{gillespiediadntc}, allows us to prove the next result. 

\begin{proposition}\label{diag2nt}
 Let $C$ be a diagonally $(X,2)$-neighbour transitive code in $H(m,q)$. Then one of the following holds:
 \begin{enumerate}
  \item $q=2$ and $C=\{(a,\ldots,a)\}$;
  \item $m=3$ or $q=2$, and $C=\Rep(m,q)$;
  \item $C=\Inj(3,q)$;
  \item $m$ is odd and $C=W([m/2],2)$; or,
  \item $q=2$ or $q=m=3$, and there exists some $p$ such that $m=pq$ and $C$ is a subset of $\All(pq,q)$. 
 \end{enumerate}
\end{proposition}

\begin{proof}
 By \cite[Theorem~4.3]{gillespiediadntc}, a diagonally neighbour transitive code $C$ is one of: $\{(a,\ldots,a)\}$ for some $a\in Q$, $\Rep(m,q)$, $\Inj(m,q)$ with $m<q$, $W([m/2],2)$ with $m$ odd, or there exists a $p$ such that $m=pq$ and $C$ is a subset of $\All(pq,q)$. Here we consider $m\geq 2$, since if $m=1$ then $C_2$ is empty, so $C$ is not $(X,2)$-neighbour transitive. Also to prove some $C$ is $(X,2)$-neighbour transitive, we need only find some $X\leq \Aut(C)$ such that $X\leq \Diag_m(S_q)\rtimes L$ and $X$ is transitive on $C_2$, since $C$ is already $X$-neighbour transitive, for some $X$, by \cite[Theorem~4.3]{gillespiediadntc}.
 
 First, if $C=\Inj(2,q)$ then $C_2$ is empty. Thus, $C$ is not $(X,2)$-neighbour transitive. Table~\ref{musandnus} lists the remaining cases which are not $2$-neighbour transitive. The second and third columns give a pair $\mu,\nu\in C_2$ such that $\Num(\mu)\neq\Num(\nu)$.  Hence, by Lemma~\ref{numpres}, $X$ is not transitive on $C_2$. It can be deduced from $\Num(\mu),\Num(\nu)$ that $\mu,\nu\in C_2$, since this makes it clear that we must change $\mu,\nu$ in at least two entries to get a vertex in $C$. Note that we let $\alpha=(1,2,3,\ldots,q)\in H(q,q)$ and in the second last and last rows we assume $\alpha\in C$ and $(\alpha,\ldots,\alpha)\in C$, respectively, and observe for the last row $\hat\mu=(1,1,1,4,5,\ldots,q)$, $\hat\nu=(1,1,3,4,5,\ldots,q)$ are in $\Gamma_2(\alpha)$.
 
\noindent
\begin{table}[ht]
  \begin{center}
   \begin{tabular}{|ll|ll|ll|}
   \hline
   \multicolumn{2}{|l|}{$C$} &   \multicolumn{2}{|l|}{$\mu\in C_2$} & \multicolumn{2}{|l|}{$\nu\in C_2$}\\
   \multicolumn{2}{|r|}{\quad\qquad Conditions} & \multicolumn{2}{|r|}{$\Num(\mu)$} & \multicolumn{2}{|r|}{$\Num(\nu)$}\\
   \hline\hline
   \multicolumn{2}{|l|}{$\{(a,\ldots,a)\}$} &  \multicolumn{2}{|l|}{$(b,b,a,\ldots,a)$} & \multicolumn{2}{|l|}{$(b,c,a,\ldots,a)$} \\
   \multicolumn{2}{|r|}{$q\geq 3$} & \multicolumn{2}{|r|}{$\{(m-2,1),(2,1)\}$} & \multicolumn{2}{|r|}{$\{(m-2,1),(1,2)\}$}\\
   \hline
   \multicolumn{2}{|l|}{$\Rep(m,q)$} & \multicolumn{2}{|l|}{$(2,2,1,\ldots,1)$} & \multicolumn{2}{|l|}{$(2,3,1,\ldots,1)$}\\
    \multicolumn{2}{|r|}{$m >q\geq 3$} & \multicolumn{2}{|r|}{$\{(m-2,1),(2,1)\}$} & \multicolumn{2}{|r|}{$\{(m-2,1),(1,2)\}$}\\
   \hline
   \multicolumn{2}{|l|}{$\Inj(m,q)$} & \multicolumn{2}{|l|}{$(1,1,1,4,5,\ldots,m)$} & \multicolumn{2}{|l|}{$(1,1,3,3,5,6,\ldots,m)$}\\
   \multicolumn{2}{|r|}{$m\geq 4$} & \multicolumn{2}{|r|}{$\{(3,1),(1,m-3)\}$} & \multicolumn{2}{|r|}{$\{(2,2),(1,m-4)\}$}\\
   \hline
   \multicolumn{2}{|l|}{$\subseteq \All(q,q)$} & \multicolumn{2}{|l|}{$(1,1,1,4,5,\ldots,q)$} & \multicolumn{2}{|l|}{$(1,1,3,3,5,6,\ldots,q)$}\\
   \multicolumn{2}{|r|}{$q\geq 4$} & \multicolumn{2}{|r|}{$\{(3,1),(1,q-3)\}$} & \multicolumn{2}{|r|}{$\{(2,2),(1,q-4)\}$}\\
   \hline
   \multicolumn{2}{|l|}{$\subseteq\All(pq,q)$} & \multicolumn{2}{|l|}{$(\hat\mu,\alpha,\ldots,\alpha)$} & \multicolumn{2}{|l|}{$(\hat\nu,\hat\nu,\alpha,\ldots,\alpha)$}\\
   \multicolumn{2}{|r|}{$q>p\geq 2$} & \multicolumn{2}{|r|}{\qquad\qquad$\{(p-1,2),(p,q-3),(p+2,1)\}$} & \multicolumn{2}{|r|}{\qquad\qquad$\{(p-2,1),(p,q-2),(p+2,1)\}$}\\
   \hline
  \end{tabular}
  \caption{\small Diagonally neighbour transitive codes $C$ which are not diagonally $2$-neighbour transitive, and elements of $C_2$ which illustrate this. Note: $\hat\mu=(1,1,1,4,5,\ldots,q)$, $\hat\nu=(1,1,3,4,5,\ldots,q)$ and $\alpha=(1,2,3,\ldots,q)$.} 
  \label{musandnus}
 \end{center}
\end{table}

 Now we prove the result for the cases which are $2$-neighbour transitive. Suppose $C=\{(a,\ldots,a)\}$ for some $a\in Q$. Let $q=2$ and $Q=\{0,1\}$. Then $L=S_m=\Aut(C)$. Without loss of generality, let $a=0$ so that $C_2$ is the set of weight two vertices. Since $L$ is transitive on the sets of weight $2$ and weight $1$ vertices, it follows $C$ is diagonally $(X,2)$-neighbour transitive. Let $C=\Rep(m,q)$. It follows from Example~\ref{k3} that $\Rep(3,q)$ is $(\Diag_3(S_q)\rtimes S_3,2)$-neighbour transitive. If $q=2$ then $\Aut(C)\cong \Diag_m(S_2)\rtimes S_m$ and $C$ is completely transitive \cite[Example~3.1]{Giudici1999647}. Consider $C=\Inj(m,q)$ with $3=m<q$ and $q\geq 4$. If $\nu\in C_2$ then $\nu_1=\nu_2=\nu_3$, since otherwise $\nu\in C$ or $C_1$. Since $\Diag_m( S_q)\leq \Aut(C)$, we are transitive on $C_2$. Suppose $C=W([m/2],2)$ and $m$ is odd. Then by \cite[Corollary~3.4]{gillespiediadntc} $C$ is $\Diag(S_2)\rtimes S_m$-completely transitive. Finally, suppose $C$ is a subset of $\All(pq,q)$ for some $p$ such that $m=pq$. Let $p\geq 2$, $q=2$ and $C=\All(2p,2)$. Then $C_2$ is the set of all weight $p\pm 2$ vertices, which $\Diag_2(S_2)\rtimes S_m\leq \Aut(C)$ is transitive on. Let $p=1$, $q=3$ and $C=\All(3,3)$. Then $C_2=\Rep(3,q)$ and is $\Aut(C)$-completely transitive by Example~\ref{k3}. 
\end{proof}

With our classification of diagonally $(X,2)$-neighbour transitive codes from the previous result, Propositions~\ref{partimpdiag} and \ref{partx2nt} mean we are now in a position to prove the main theorem.

\begin{proof}[Proof of Theorem~\ref{aasmain}]
 Suppose $C$ is an alphabet-almost-simple $(X,2)$-neighbour transitive code with $\delta\geq 3$ such that $X\cap B\neq 1$. By Proposition~\ref{partimpdiag}, there exists an $X$-invariant partition $\J=\{J_1,\ldots,J_\ell\}$, for some $\ell$, for the action of $X$ on $M$. Moroever, $\pi_{J_i}(C)$ has minimum distance at least $2$ and is diagonally $\chi_{J_i}(X)$-neighbour transitive. By Proposition~\ref{partx2nt}, either $\pi_{J_i}(C)$ has covering radius $\rho\leq 1$, or $\pi_{J_i}(C)$ is also $(\chi_{J_i}(X),2)$-neighbour transitive. Note $\rho\neq 0$, that is, $\pi_{J_i}(C)$ is not the complete code, since $\pi_{J_i}(C)$ has minimum distance at least $2$.
 
 Suppose $\pi_{J_i}(C)$ has covering radius $\rho\geq 2$. Since $X_1^Q$ is almost-simple, it follows that $q\geq 5$. By Proposition~\ref{diag2nt}, the only diagonally $2$-neighbour transitive code with $q\geq 5$ and $\delta\geq 2$ is $\Rep(3,q)$ for $q\geq 5$ (note that $\delta=1$ for $\Inj(3,q)$). Then Lemma~\ref{repnopart} implies $\J$ is a trivial partition. Since $|J_i|=k=3>1$, it follows that $\ell=1$, $k=m$, and $C=\Rep(3,q)$.
 
 Suppose $\pi_{J_i}(C)$ has covering radius $\rho=1$. By \cite[Thm.~4 and Cor.~2]{gillespiediadntc}, the only diagonally neighbour transitive code with $\delta\geq 2$ and $\rho=1$ is $\Rep(2,q)$. If $l=1$ then $\delta=2$, a contradiction. Suppose $l\geq 2$. Then Lemma~\ref{repnopart} implies $\delta=2$, a contradiction.
\end{proof}

\end{document}